\documentclass[12pt]{article}
\usepackage{amsthm}
\usepackage{amsmath,amsfonts,amssymb}
\usepackage{color,graphics}
\usepackage{indentfirst}
\usepackage{graphicx}
\usepackage{multirow}
\usepackage{geometry}
\usepackage{titlesec}
\usepackage{caption}
\usepackage{authblk}
\usepackage[misc]{ifsym}  
\usepackage{skak} 
\newtheorem{dingyi}{Definition}[section]
\newtheorem{yinli}[dingyi]{Lemma}
\newtheorem{dingli}[dingyi]{Theorem}

\newtheorem{lizi}{Example}

\newtheorem{caixiang}[]{Conjecture}

\newcommand{\spa}{\mathrm{span}}
\newcommand{\bol}{Bollob\'as}
\newcommand{\lov}{Lov\'asz}
\newcommand{\fu}{F\"{u}redi}
\newcommand{\ekr}{Erd\"{o}s-Ko-Rado}
\newcommand{\heg}{Heged\"{u}s}

\newcommand{\ELTE}{E\"{o}tv\"{o}s Lor\'{a}nd University}
\newcommand{\ELTEAdress}{P\'{a}zm\'{a}ny P\'{e}ter s\'{e}t\'{a}ny 1/C, Budapest, Hungary, H-1117}

\title{On~\bol-type theorems of $d$-tuples}

\author[$\ast\sympawn$]{Erfei Yue}

\affil[$\sympawn$]{\footnotesize Institute of Mathematics, \ELTE,

\ELTEAdress,

\Letter\ yef9262@mail.bnu.edu.cn}

\date{}

\begin{document}
\maketitle

\begin{center}
\textbf{Abstract}
\end{center}
In 1965,~\bol~\cite{Sets} proved that for a~\bol~set-pair system~$\{(A_i,B_i)\mid i\in[m]\}$, the maximum value of~$\sum_{i=1}^m\binom{|A_i|+|B_i|}{A_i}^{-1}$ is~$1$.
~\heg~and Frankl~\cite{Hegedus} recently extended the concept of~\bol~systems to~$d$-tuples,
conjecturing that for a~\bol~system of~$d$-tuples,~$\{(A_i^{(1)},\ldots,A_i^{(d)})\mid i\in[m]\}$,
the maximum value of~$\sum_{i=1}^m\binom{|A_i^{(1)}|+\cdots+|A_i^{(d)}|}{|A_i^{(1)}|,\ldots,|A_i^{(d)}|}^{-1}$ is also~$1$.
This paper refutes this conjecture and establishes an upper bound for the sum.
In the case~$d=3$, the derived upper bound is asymptotically tight. Furthermore, we sharpen an inequality for skew~\bol~systems of~$d$-tuples in~\cite{Hegedus},
Finally, we determine the maximum size of a uniform skew~\bol~system of~$d$-tuples on both sets and spaces.

\section{Introduction}

Let~$[n]=\{1,\ldots,n\}$ be a set of size~$n$.
The following concepts of~\bol~systems and skew~\bol~systems were introduced by~\bol~\cite{Sets} and Frankl~\cite{Skew} respectively.

\begin{dingyi}
Suppose~$\mathcal{P}=\{(A_i,B_i)\mid i\in[m]\}$ is a family of pairs of sets, where~$A_i,B_i\subseteq [n]$, and~$A_i\cap B_i=\emptyset$.
Then~$\mathcal{P}$ is called a~\emph{\bol~system} if~$A_i\cap B_j\neq\emptyset,\forall i\neq j$,
and a \emph{skew~\bol~system} if~$A_i\cap B_j\neq\emptyset,\forall i<j$.
\end{dingyi}

To solve a problem on hypergraphs,~\bol~\cite{Sets} proved the following theorem in 1965.
In 1975, Tarj\'an~\cite{Tarjan} independently provided a more elegant proof.

\begin{dingli}[\bol~\cite{Sets}]\label{Th:nonuniform}
Let~$\mathcal{P}=\{(A_i,B_i)\mid i\in[m]\}$ be a~\bol~system, where~$A_i,B_i\subseteq [n]$. Then
\begin{equation*}
\sum_{i=1}^m \binom{|A_i|+|B_i|}{|A_i|}^{-1}\leqslant 1.
\end{equation*}
\end{dingli}

Theorem~\ref{Th:nonuniform} also gives the maximum size of a uniform~\bol~system. Specifically, if~$|A_i|=a,|B_i|=b$ for every~$i\in[m]$,
Theorem~\ref{Th:nonuniform} implies that~$m\leqslant\binom{a+b}{a}$.
This upper bound was later proved by Jaeger and Payan~\cite{FRA}, and Katona~\cite{Katona} independently with different methods.
In 1982, using polynomial method, Frankl~\cite{Skew} proved that this upper bound remains true for skew~\bol~systems.
This generalization has application in automata theory~\cite{Pin}.

\begin{dingli}[Frankl~\cite{Skew}]\label{Th:skew-uniform}
Let~$\mathcal{P}=\{(A_i,B_i)\mid i\in[m]\}$ be a skew~\bol~system, where~$A_i,B_i\subseteq [n]$, and~$|A_i|=a,|B_i|=b$ for any~$i$.
Then~$m\leqslant\binom{a+b}{a}$.
\end{dingli}

While Theorem~\ref{Th:nonuniform} no longer hold if we replace the condition~\bol~system by skew~\bol~system.
The following (tight) inequalities for skew~\bol~systems (instead of~\bol~systems) are proved by~\heg~and Frankl~\cite{Hegedus}, and Yue~\cite{Yue}.

\begin{dingli}[Heged\"{u}s and Frankl~\cite{Hegedus}, Yue~\cite{Yue}]\label{Th:skew-nonuniform}
Suppose~$A_i,B_i\subseteq [n],i\in [m]$, and~$\mathcal{P}=\{(A_i,B_i)\mid i\in[m]\}$ is a skew~\bol~system. Then we have
\begin{equation*}
\sum_{i=1}^m \left((1+|A_i|+|B_i|)\binom{|A_i|+|B_i|}{|A_i|}\right)^{-1}\leqslant 1,
\end{equation*}
and so
\begin{equation*}
\sum_{i=1}^m \binom{|A_i|+|B_i|}{|A_i|}^{-1}\leqslant 1+n.
\end{equation*}
\end{dingli}

Note that there are two kinds of structures:~\bol~systems and skew~\bol~systems; and for each of them, there are two questions to ask:
what is the maximum size of the family in the uniform case, and what inequality can we get in the nonuniform case.
The following table is a brief summary so far.

\begin{table}[h]
\centering
\begin{tabular}{|l|c|c|}
\hline
 & \bol~systems & Skew~\bol~systems \\
\hline
Uniform cases & \bol~\cite{Sets}& \heg~and Frankl~\cite{Hegedus} and Yue~\cite{Yue} \\
\hline
Nonuniform cases & \bol~\cite{Sets} & Frankl~\cite{Skew} \\
\hline
\end{tabular}
\caption{A brief map for~\bol-type theorem}
\end{table}

A natural generalization of (skew)~\bol~systems involves considering families of~$d$-tuples instead of set-pairs.
In~\cite{Hegedus},~\heg~and Frankl defined~\bol~systems and skew~\bol~systems for~$d$-tuples as follows.
Note that there are some other ways to define these concepts on~$d$-tuples, see~\cite{kTuples} for instance.

\begin{dingyi}
Let~$\mathcal{F}=\{(A_i^{(1)},\ldots,A_i^{(d)})\mid i\in[m]\}$ be a collection of~$d$-tuples. Suppose that~$A_i^{(p)}\cap A_i^{(q)}=\emptyset$
for any~$i\in[m]$ and~$p\neq q$. Then~$\mathcal{F}$ is a~\emph{\bol~system (skew~\bol~system)} if for any~$i\neq j$ \emph{(}$i<j$\emph{)},
there exist~$p<q$ such that~$A_i^{(p)}\cap A_j^{(q)}\neq\emptyset$.
\end{dingyi}

For a~$d$-tuple~$(A^{(1)},\ldots,A^{(d)})$, we say~$(|A^{(1)}|,\ldots,|A^{(d)}|)$ is its \emph{type}. Following is an example of~\bol~system of~$d$-tuples.

\begin{lizi}\label{Ex:sym}
For fixed~$a_1,\ldots,a_d$, let~$\mathcal{F}$ be the collection of all disjoint~$d$-tuples~$(A^{(1)},\ldots,A^{(d)})$ of type~$(a_1,\ldots,a_d)$
such that~$A^{(k)}\in [a_1+\cdots+a_d],\forall k\in[d]$. For distinct~$(A^{(1)},\ldots,A^{(d)})$ and~$(B^{(1)},\ldots,B^{(d)})$,
there is a~$k$ such that~$A^{(k)}\neq B^{(k)}$.
Then we have~$A^{(k)}\cap\left(\cup_{l\neq k}B^{(l)}\right)\neq\emptyset$ and~$B^{(k)}\cap\left(\cup_{l\neq k}A^{(l)}\right)\neq\emptyset$.
Hence~$\mathcal{F}$ is a~\bol~system of~$d$-tuples.
\end{lizi}

For a fixed~$n\geqslant 0$, and~$k_1,\ldots,k_t\leqslant n$, define
\begin{equation*}
\binom{n}{k_1,\ldots,k_t}:=\frac{n!}{k_1!\cdots k_t!(n-k_1-\cdots-k_t)!},
\end{equation*}
which is a natural generalization of the binomial coefficients.
Then the size of the~\bol~system defined in Example~\ref{Ex:sym} is~$\binom{a_1+\cdots+a_d}{a_1,\ldots,a_d}$.

To generalize Theorem~\ref{Th:nonuniform},~\heg~and Frankl~\cite{Hegedus} proposed the following conjecture.

\begin{caixiang}[Heged\"{u}s and Frankl~\cite{Hegedus}] \label{Conj}
Suppose that~$\mathcal{F}=\{(A_i^{(1)},\ldots,A_i^{(d)})\mid i\in[m]\}$ is a~\bol~system of~$d$-tuples. Then
\begin{equation*}
\sum_{i=1}^m{\binom{|A_i^{(1)}|+\cdots+|A_i^{(d)}|}{|A_i^{(1)}|,\ldots,|A_i^{(d)}|}}^{-1}\leqslant 1.
\end{equation*}
\end{caixiang}

However, we can deny this conjecture with the following example.

\begin{lizi}\label{Ex:counter}
Fix an arbitrary~$n$, and~$d=3$. For~$l=0,\dots,\lfloor\frac{n}{2}\rfloor$,
let~$\mathcal{F}_l$ be the collection of all disjoint triples~$(A,B,C)$ of type~$(l,n-2l,l)$, such that~$A,B,C\subseteq [n]$,
and~$\mathcal{F}$ be the union of~$\mathcal{F}_l$'s. We claim that~$\mathcal{F}$ is a~\bol~system.
Actually, for two different families~$(A,B,C),(A',B',C')\in\mathcal{F}_l$, if~$A=A'$, then~$B\cap C'\neq\emptyset$ and~$B'\cap C\neq\emptyset$.
If~$A\neq A'$, then~$A\cap(B'\cup C')=\neq\emptyset$ and~$A'\cap(B\cup C)\neq\emptyset$.
For~$(A,B,C)\in\mathcal{F}_l$ and~$(A',B',C')\in\mathcal{F}_{l'}$ where~$l<l'$, note that~$|A|+|B|+|C'|=|A'|+|B|+|C|=n-l+l'>n$,
so~$A\cap(B'\cup C')\neq\emptyset$ and~$A'\cap(B\cup C)\neq\emptyset$. Hence~$\mathcal{F}$ is a~\bol~system. But we have
\begin{equation*}
\sum_{(A,B,C)\in\mathcal{F}}{\binom{|A|+|B|+|C|}{|A|,|B|,|C|}}^{-1}
=\sum_{l=0}^{\lfloor\frac{n}{2}\rfloor}\sum_{(A,B,C)\in\mathcal{F}_l}{\binom{|A|+|B|+|C|}{|A|,|B|,|C|}}^{-1}=\left\lfloor\frac{n}{2}\right\rfloor+1.
\end{equation*}
\end{lizi}

It is important to note that Tian and Wu~\cite{Tian} also notice the invalidity of this conjecture, as well as Theorem~\ref{Th:d-skew-nonuniform}.
For further details, please refer to the remark at the end of this paper.

Since Conjecture~\ref{Conj} is not true, a natural subsequent question arises: for a~\bol~system~$\{(A_i^{(1)},\ldots,A_i^{(d)})\mid i\in[m]\}$ of~$d$-tuples,
what is the maximum value of~$\sum_{i=1}^m\binom{|A_i^{(1)}|+\cdots+|A_i^{(d)}|}{|A_i^{(1)}|,\ldots,|A_i^{(d)}|}^{-1}$?
Generalizing the probabilistic method in~\cite{Yue}, we can prove the following inequalities for~\bol~systems of~$d$-tuples.
The inequality for~$d=3$ shows the construction in Example~\ref{Ex:counter} is nearly extremum, and is asymptotically tight as~$n\rightarrow\infty$.
For the case~$d>3$, it is still not clear if the inequality is tight or not.

\begin{dingli} \label{Th:d-nonuniform}
Suppose that~$\mathcal{F}=\{(A_i^{(1)},\ldots,A_i^{(d)}\mid i\in[m])\}$ is a~\bol~system of~$d$-tuples,
and~$A_i^{(p)}\subseteq[n]$ for all~$i\in[m], p\in[d]$. Then for the case~$d=3$ we have
\begin{equation*}
\sum_{i=1}^m{\binom{|A_i^{(1)}|+|A_i^{(2)}|+|A_i^{(3)}|}{|A_i^{(1)}|,|A_i^{(2)}|,|A_i^{(3)}|}}^{-1}\leqslant\frac{n+3}{2},
\end{equation*}
and for arbitrary~$d$ we have
\begin{equation*}
\sum_{i=1}^m{\binom{|A_i^{(1)}|+\cdots+|A_i^{(d)}|}{|A_i^{(1)}|,\ldots,|A_i^{(d)}|}}^{-1}\leqslant \frac{1}{d-1}\binom{n+d-2}{d-2}+O(n^{d-3})
\end{equation*}
as~$n\rightarrow\infty$.
\end{dingli}

For skew~\bol~systems of~$d$-tuples,~\heg~and Frankl~\cite{Hegedus} proved the following theorem.

\begin{dingli}[\heg~and Frankl~\cite{Hegedus}] \label{Th:FH}
Suppose that~$\mathcal{F}=\{(A_i^{(1)},\ldots,A_i^{(d)})\mid i\in[m]\}$ is a skew~\bol~system of~$d$-tuples,
and~$A_i^{(p)}\subseteq[n]$ for all~$i\in[m]$ and~$p\in[d]$. Then
\begin{equation*}
\sum_{i=1}^m{\binom{|A_i^{(1)}|+\cdots+|A_i^{(d)}|}{|A_i^{(1)}|,\ldots,|A_i^{(d)}|}}^{-1}\leqslant\binom{n+d-1}{d-1}.
\end{equation*}
\end{dingli}

Although this result is already tight, using similar probabilistic argument, we can still improve it as follows.

\begin{dingli} \label{Th:d-skew-nonuniform}
Suppose that~$\mathcal{F}=\{(A_i^{(1)},\ldots,A_i^{(d)})\mid i\in[m]\}$ is a skew~\bol~system of~$d$-tuples. Then
\begin{equation*}
\sum_{i=1}^m\left(\binom{|A_i^{(1)}|+\cdots+|A_i^{(d)}|+d-1}{d-1}\binom{|A_i^{(1)}|+\cdots+|A_i^{(d)}|}{|A_i^{(1)}|,\ldots,|A_i^{(d)}|}\right)^{-1}\leqslant 1.
\end{equation*}
\end{dingli}

Note that~$|A_i^{(1)}|+\cdots+|A_i^{(d)}|\leqslant n$, so Theorem~\ref{Th:d-skew-nonuniform} implies Theorem~\ref{Th:FH}.
Theorem~\ref{Th:d-skew-nonuniform} also provides an upper bound for the size of uniform skew~\bol~systems.
If we further assume that~$|A_i^{(p)}|=a_p,\forall i\in[m],p\in[d]$,
then we have~$m\leqslant\binom{n+d-1}{d-1}\binom{a_1+\cdots+a_d}{a_1,\ldots,a_d}$ where~$n$ denotes the size of the ground set.
However, this upper bound is not so good. Actually, we can prove the following upper bound, which is confirmed to be tight by Example~\ref{Ex:sym}.
(Note that a~\bol~system is also a skew~\bol~system.)

\begin{dingli} \label{Th:d-skew-uniform}
Suppose that~$\mathcal{F}=\{(A_i^{(1)},\ldots,A_i^{(d)})\mid i\in[m]\}$ is a skew~\bol~system of~$d$-tuples, and~$|A_i^{(p)}|=a_p,\forall i\in[m],p\in[d]$. Then
\begin{equation*}
m\leqslant\binom{a_1+\cdots+a_d}{a_1,\ldots,a_d}.
\end{equation*}
\end{dingli}

Like Table 1, the following table give a brief summary for~\bol-type theorem on~$d$-tuples.

\begin{table}[h]
\centering
\begin{tabular}{|l|c|c|}
\hline
 & \bol~systems & Skew~\bol~systems \\
\hline
 Uniform & Solved: Theorem~\ref{Th:d-skew-uniform} & Solved: Theorem~\ref{Th:d-skew-uniform} \\
\hline
\multirow{2}*{Nonuniform} & Conjecture:~\heg~and Frankl~\cite{Hegedus} & Solved:~\heg~and Frankl~\cite{Hegedus} \\
~ & Partly solved: Example~\ref{Ex:counter}$^\star$ and Theorem~\ref{Th:d-nonuniform} & Improved: Theorem~\ref{Th:d-skew-nonuniform}$^\star$ \\
\hline
\multicolumn{3}{|l|}{$\star$: Tian and Wu~\cite{Tian} also have contribution on these two questions.} \\
\hline
\end{tabular}
\caption{A brief map for~\bol-type theorem on~$d$-tuples}
\end{table}

\newpage

As an analog,~\bol~systems and skew~\bol~systems of vector spaces are defined as follows.

\begin{dingyi}
Let~$\mathcal{P}=\{(A_i,B_i)\mid i\in[m]\}$ be a family of pairs of subspaces of a fixed vector space~$V$, where~$\dim(A_i\cap B_i)=0$.
We say that~$\mathcal{P}$ is a\emph{~\bol~system} if~$\dim(A_i\cap B_j)>0,\forall i\neq j$,
and a \emph{skew~\bol~system} if~$\dim(A_i\cap B_j)>0,\forall i<j$.
\end{dingyi}

In 1977, employing powerful algebraic techniques,~\lov~\cite{LinearSpaces,Lov1,Lov2} extended Theorem~\ref{Th:skew-uniform} to the domain of matroids.
This seminal extension not only opened a new avenue for solving combinatorial problems using the exterior product method,
but also elevated the~\bol-type theorem to a central position within the field of extremal set theory.
Perhaps for the sake of simplicity,~\fu~\cite{Threshold} reformulated it as follows for real vector spaces.

\begin{dingli}[\lov~\cite{LinearSpaces,Lov1,Lov2}]\label{Th:Lovasz}
Let~$\mathcal{P}=\{(A_i,B_i)\mid i\in[m]\}$ be a (skew)~\bol~system of subspaces of~$\mathbb{R}^n$.
If~$\dim(A_i)=a,\dim(B_i)=b$ for all~$i$, then~$m\leqslant\binom{a+b}{a}$.
\end{dingli}

The conclusion of Theorem~\ref{Th:Lovasz} holds for vector spaces over an arbitrary field.
For the case of~$\mathbb{Q}^n$ or~$\mathbb{C}^n$, the proofs remain essentially the same. For other fields, such as finite fields,
some technical adaptations may be necessary. More details about this can be found in~\cite{Book}.
After~\fu, other scholars~(such as Scott and Wilmer~\cite{Scott}, and Yu, Kong, Xi, Zhang, and Ge~\cite{HemiBundled}) have tended to express their theorems in the same way,
even though they may hold more generally.

\bol-type theorems and their variations have attracted significant interest among mathematicians.
In 1984,~\fu~\cite{Threshold} proved the threshold (or~$t$-intersecting) version of~\bol-type theorem (for both sets and spaces).
This work laid the foundation for further extensions by Zhu~\cite{tl}, Talbot~\cite{Inequality}, and Kang, Kim, and Kim~\cite{Inequality2}.
Recent research has explored even more diverse variations of~\bol-type theorems,
such as those concerning partitions of sets~\cite{Alon}, affine spaces~\cite{AffineSpaces}, and weakly~\bol~systems~\cite{weakly,Tuza}.

It is natural to extend the concept of~\bol~systems and skew~\bol~systems to~$d$-tuples of spaces.

\begin{dingyi}
Let~$\mathcal{F}=\{(A_i^{(1)},\ldots,A_i^{(d)})\mid i\in[m]\}$, where~$A_i^{(k)}\leqslant V=\mathbb{R}^n$.
Suppose that~$\dim(A_i^{(1)}+\cdots+A_i^{(d)})=\dim(A_i^{(1)})+\cdots+\dim(A_i^{(d)})$ for any~$i\in[m]$.
Then~$\mathcal{F}$ is a~\emph{\bol~system (skew~\bol~system)} if for any~$i\neq j$ \emph{(}$i<j$\emph{)}, there exist~$p<q$ such that~$A_i^{(p)}\cap A_j^{(q)}\neq\{0\}$.
\end{dingyi}

Our last result is the following theorem, which determines the maximum size of a skew~\bol~system of~$d$-tuples of spaces.
This theorem strengthens Theorem~\ref{Th:d-skew-uniform}, and generalizes Theorem~\ref{Th:Lovasz}.

\begin{dingli}\label{Th:d-space}
Let~$\mathcal{F}=\{(A_i^{(1)},\ldots,A_i^{(d)})\mid i\in[m]\}$ be a skew~\bol~system of~$d$-tuples of spaces, where~$A_i^{(k)}\leqslant V=\mathbb{R}^n$.
Suppose that~$\dim(A_i^{(k)})=a_k$ for every~$i\in[m]$ and~$k\in[d]$. Then we have
\begin{equation*}
m\leqslant\binom{a_1+\cdots+a_d}{a_1,\ldots,a_d}.
\end{equation*}
\end{dingli}

This paper is organized as follows: in Section 2, we prove Theorem~\ref{Th:d-nonuniform} and Theorem~\ref{Th:d-skew-nonuniform},
in Section 3, we prove Theorem~\ref{Th:d-space}, and obtain Theorem~\ref{Th:d-skew-uniform} from it.

\section{Proof of Theorem~\ref{Th:d-nonuniform} and Theorem~\ref{Th:d-skew-nonuniform}}

For a fixed~$n>0$, denote~$S_n$ the group of all permutations on~$[n]$.
For a~$\sigma\in S_n$ and~$A\subseteq [n]$, denote~$\sigma(A):=\{\sigma(a)\mid a\in A\}$ the image of~$A$ over~$\sigma$.
For~$\sigma\in S_n$ and~$A,B\subseteq [n]$, obviously~$\sigma(A)\cap\sigma(B)=\emptyset$ if and only if~$A\cap B=\emptyset$.
Further more, for a family of~$d$-tuples~$\mathcal{F}=\{(A_i^{(1)},\ldots,A_i^{(d)})\mid i\in[m]\}$, where~$A_i^{(k)}\subseteq [n],\forall i\in[m],k\in[d]$,
and any~$\sigma\in S_n$,
\begin{equation*}
\mathcal{F}'=\{(\sigma(A_i^{(1)}),\ldots,\sigma(A_i^{(d)}))\mid i\in[m]\}
\end{equation*}
is a (skew)~\bol~system, if and only if~$\mathcal{F}$ is a (skew)~\bol~system.
We will prove Theorem~\ref{Th:d-nonuniform} and Theorem~\ref{Th:d-skew-nonuniform} with the tool of random permutations, and begin with Theorem~\ref{Th:d-skew-nonuniform}.

\begin{proof}[Proof of Theorem~\ref{Th:d-skew-nonuniform}]
Without lose of generality, we assume~$A_i^{(k)}\subseteq[n],\forall i\in[m],k\in[d]$ for some~$n$.
Pick a random permutation~$\sigma\in S_{n+d-1}$ uniformly, and it gives the~$n+d-1$ elements a new order.
We will use the elements in~$\{n+1,\ldots,n+d-1\}$ as delimiters.
Let~$E_i$ be the event that elements in~$\sigma(A_i^{(1)}),\ldots,\sigma(A_i^{(d)})$ appear in order,
and each pair of neighbors are divided by a delimiter. Formally speaking, define
\begin{equation*}
\begin{split}
  E_i:= & \{\sigma\in S_{n+d-1}\mid \sigma(a_1)<\sigma(b_1)<\sigma(a_2)<\cdots<\sigma(b_{d-1})<\sigma(a_d),\forall a_k\in A_i^{(k)}, \\
    & \textrm{where}~\{b_1,\ldots,b_{d-1}\}=\{n+1,\ldots,n+d-1\}\}.
\end{split}
\end{equation*}
To calculate the probability of~$E_i$, we only need to consider about elements in~$A_i^{(1)}\cup\cdots\cup A_i^{(d)}\cup\{n+1,\ldots,n+d-1\}$.
We put these elements into the~$|A_i^{(1)}|+\cdots+|A_i^{(d)}|+d-1$ slots.
There are~$\binom{|A_i^{(1)}|+\cdots+|A_i^{(d)}|+d-1}{d-1}$ many of places to put the last~$d-1$ elements, regardless of their inner order,
and exactly one of them is what we needed. After fixed the elements in~$\{n+1,\ldots,n+d-1\}$, the rest~$|A_i^{(1)}|+\cdots+|A_i^{(d)}|$ slots are divided into~$d$ parts.
We need to put each of~$A_i^{(k)}$ into the correct place, and the probability of it is
\begin{equation*}
\binom{|A_i^{(1)}|+\cdots+|A_i^{(d)}|}{|A_i^{(1)}|}^{-1}\binom{|A_i^{(2)}|+\cdots+|A_i^{(d)}|}{|A_i^{(2)}|}^{-1}\cdots\binom{|A_i^{(d-1)}|+|A_i^{(d)}|}{|A_i^{(d-1)}|}^{-1}
=\binom{|A_i^{(1)}|+\cdots+|A_i^{(d)}|}{|A_i^{(1)}|,\ldots,|A_i^{(d)}|}^{-1},
\end{equation*}
and so
\begin{equation*}
\mathbb{P}(E_i)=\left(\binom{|A_i^{(1)}|+\cdots+|A_i^{(d)}|+d-1}{d-1}\binom{|A_i^{(1)}|+\cdots+|A_i^{(d)}|}{|A_i^{(1)}|,\ldots,|A_i^{(d)}|}\right)^{-1}.
\end{equation*}
Now it is enough to show~$E_i$ and~$E_j$ cannot happen at the same time if~$i\neq j$. Actually, if~$\sigma\in E_i\cap E_j$, for every~$p<q$,
there is a~$b\in\{n+1,\ldots,n+d-1\}$ such that~$\sigma(a_i)<\sigma(b)<\sigma(a_j),\forall a_i\in A_i^{(p)}, a_j\in A_j^{(q)}$, so~$A_i^{(p)}\cap A_j^{(q)}=\emptyset$.
For the same reason, we have~$A_i^{(q)}\cap A_j^{(p)}=\emptyset$. Then we get a contradiction, no matter~$i<j$ or~$i>j$. Therefore, we have
\begin{equation*}
1\geqslant\mathbb{P}\left(\bigcup_{i=1}^m E_i\right)=\sum_{i=1}^m\mathbb{P}(E_i)
=\sum_{i=1}^m\left(\binom{|A_i^{(1)}|+\cdots+|A_i^{(d)}|+d-1}{d-1}\binom{|A_i^{(1)}|+\cdots+|A_i^{(d)}|}{|A_i^{(1)}|,\ldots,|A_i^{(d)}|}\right)^{-1}.
\end{equation*}
\end{proof}

Now we are going to prove Theorem~\ref{Th:d-nonuniform} with similar probabilistic argument.

\begin{proof}[Proof of Theorem~\ref{Th:d-nonuniform}]
We prove the theorem by induction on~$d$. For~$d=3$, pick a random permutation~$\sigma\in S_{n+1}$ uniformly.~$n+1$ will play the role of delimiter.
Let~$E_i$ and~$F_i$ be events that elements are placed in the order~$\sigma(A_i^{(1)}),\sigma(n+1),\sigma(A_i^{(2)}),\sigma(A_i^{(3)})$
and~$\sigma(A_i^{(1)}),\sigma(A_i^{(2)}),\sigma(n+1),\sigma(A_i^{(3)})$ correspondingly. In other words, let
\begin{eqnarray*}
  E_i &:=& \{\sigma\in S_{n+1}\mid \sigma(a_1)<\sigma(n+1)<\sigma(a_2)<\sigma(a_3),\forall a_k\in A_i^{(k)}\},  \\
  F_i &:=& \{\sigma\in S_{n+1}\mid \sigma(a_1)<\sigma(a_2)<\sigma(n+1)<\sigma(a_3),\forall a_k\in A_i^{(k)}\}.
\end{eqnarray*}
Then similar with the discussion in the former proof, we have
\begin{equation*}
\mathbb{P}(E_i)=\mathbb{P}(F_i)=\left((|A_i^{(1)}|+|A_i^{(2)}|+|A_i^{(3)}|+1)\binom{|A_i^{(1)}|+|A_i^{(2)}|+|A_i^{(3)}|}{|A_i^{(1)}|,|A_i^{(2)}|,|A_i^{(3)}|}\right)^{-1}.
\end{equation*}

For~$i\neq j$, we claim all of~$E_i\cap E_j$,~$F_i\cap F_j$,~$E_i\cap F_j$, and~$E_j\cap F_i$ are empty. Suppose~$\sigma\in E_i\cap E_j$.
Then~$A_i^{(1)}\cap(A_j^{(2)}\cup A_j^{(3)})=A_j^{(1)}\cap(A_i^{(2)}\cup A_i^{(3)})=\emptyset$.
By the definition of~\bol~system, neither~$A_i^{(2)}\cap A_j^{(3)}$ nor~$A_i^{(3)}\cap A_j^{(2)}$ is empty.
However,~$t\in A_i^{(2)}\cap A_j^{(3)}$ implies
\begin{equation*}
\sigma(a_i)>\sigma(t)>\sigma(a_j),\forall a_i\in A_i^{(3)},a_j\in A_j^{(2)},
\end{equation*}
and~$A_i^{(3)}\cap A_j^{(2)}=\emptyset$. Hence we get a contradiction, and so~$E_i\cap E_j=\emptyset$. Similarly,~$F_i\cap F_j=\emptyset$.
Suppose~$\sigma\in E_i\cap F_j$. Then elements in~$\sigma(A_i^{(1)}),\sigma(A_j^{(1)}),\sigma(A_j^{(2)})$ appear before~$\sigma(n+1)$,
while that in~$\sigma(A_i^{(2)}),\sigma(A_i^{(3)}),\sigma(A_j^{(3)})$ after it,
so we have~$A_i^{(2)}\cap A_j^{(1)}=A_i^{(3)}\cap A_j^{(1)}=A_i^{(3)}\cap A_j^{(2)}=\emptyset$, which is a contradiction.
Hence~$E_i\cap F_j=\emptyset$, and similarly~$E_j\cap F_i=\emptyset$.

If~$A_i^{(2)}\neq\emptyset$, then~$E_i\cap F_i=\emptyset$. Then we have
\begin{equation*}
1\geqslant\mathbb{P}\left(\bigcup_{A_i^{(2)}\neq\emptyset} (E_i\cup F_i)\right)
=2\sum_{A_i^{(2)}\neq\emptyset}\left((|A_i^{(1)}|+|A_i^{(2)}|+|A_i^{(3)}|+1)\binom{|A_i^{(1)}|+|A_i^{(2)}|+|A_i^{(3)}|}{|A_i^{(1)}|,|A_i^{(2)}|,|A_i^{(3)}|}\right)^{-1},
\end{equation*}
and so
\begin{equation*}
\sum_{A_i^{(2)}\neq\emptyset}\binom{|A_i^{(1)}|+|A_i^{(2)}|+|A_i^{(3)}|}{|A_i^{(1)}|,|A_i^{(2)}|,|A_i^{(3)}|}^{-1}
\leqslant\frac{n+1}{2}.
\end{equation*}

For~$A_i^{(2)}=\emptyset$, the collection~$\{(A_i^{(1)},A_i^{(3)})\mid A_i^{(2)}=\emptyset\}$ forms a~\bol~system (of~$d=2$).
Then by Theorem~\ref{Th:nonuniform}, we have
\begin{equation*}
\sum_{A_i^{(2)}=\emptyset}\binom{|A_i^{(1)}|+|A_i^{(2)}|+|A_i^{(3)}|}{|A_i^{(1)}|,|A_i^{(2)}|,|A_i^{(3)}|}^{-1}
=\sum_{A_i^{(2)}=\emptyset}\binom{|A_i^{(1)}|+|A_i^{(3)}|}{|A_i^{(1)}|}^{-1}\leqslant 1.
\end{equation*}

Combing these two cases, we have
\begin{equation*}
\sum_{i=1}^m\binom{|A_i^{(1)}|+|A_i^{(2)}|+|A_i^{(3)}|}{|A_i^{(1)}|,|A_i^{(2)}|,|A_i^{(3)}|}^{-1}\leqslant\frac{n+3}{2}.
\end{equation*}

Now we are going to prove the theorem for arbitrary~$d$ while assuming it is true for the case~$d-1$.
Pick a random permutation~$\sigma\in S_{n+d-2}$ uniformly. Once again, we will use the elements in~$\{n+1,\ldots,n+d-2\}$ as delimiters.
For~$i\in[m]$ and~$k\in[d-1]$, let~$E_i^{(k)}$ be the event that~$\sigma(a_1)<\cdots<\sigma(a_d),\forall a_l\in A_i^{(l)}$,
and all pairs in~$\{(\sigma(A_i^{(l)}),\sigma(A_i^{(l+1)}))\mid l\in[d-1]\}$ except~$(\sigma(A_i^{(k)}),\sigma(A_i^{(k+1)}))$ are divided by a delimiter. Then
\begin{equation*}
\mathbb{P}(E_i^{(k)})=\left(\binom{|A_i^{(1)}|+\cdots+|A_i^{(d)}|+d-2}{d-2}\binom{|A_i^{(1)}|+\cdots+|A_i^{(d)}|}{|A_i^{(1)}|,\ldots,|A_i^{(d)}|}\right)^{-1}.
\end{equation*}

Similar with the discussion in the first part of this proof, we have~$E_i^{(k)}\cap E_j^{(l)}=\emptyset$ for any~$i\neq j$,
and~$E_i^{(k)}\cap E_i^{(l)}=\emptyset,\forall k\neq l$ if none of~$A_i^{(2)},\ldots,A_i^{(d-1)}$ is empty. Then we have
\begin{equation*}
\begin{split}
  1 & \geqslant\mathbb{P}\left(\bigcup_{i:A_i^{(l)}\neq\emptyset,\forall 2\leqslant l\leqslant d-1}\bigcup_{k=1}^{d-1} E_i^{(k)}\right)
      =\sum_{i:A_i^{(l)}\neq\emptyset,\forall 2\leqslant l\leqslant d-1}\sum_{k=1}^{d-1}\mathbb{P}(E_i^{(k)})  \\
    & =(d-1)\sum_{i:A_i^{(l)}\neq\emptyset,\forall 2\leqslant l\leqslant d-1}
        \left(\binom{|A_i^{(1)}|+\cdots+|A_i^{(d)}|+d-2}{d-2}\binom{|A_i^{(1)}|+\cdots+|A_i^{(d)}|}{|A_i^{(1)}|,\ldots,|A_i^{(d)}|}\right)^{-1},
\end{split}
\end{equation*}
and so
\begin{equation*}
\sum_{i:A_i^{(k)}\neq\emptyset,2\leqslant k\leqslant d-1}\binom{|A_i^{(1)}|+\cdots+|A_i^{(d)}|}{|A_i^{(1)}|,\ldots,|A_i^{(d)}|}^{-1}
\leqslant\frac{1}{d-1}\binom{n+d-2}{d-2}.
\end{equation*}
Note that for~$k\in\{2,\ldots,d-1\}$,
the family~$\{(A_i^{(1)},\ldots,A_i^{(k-1)},A_i^{(k+1)},\ldots,A_i^{(d)})\mid A_i^{(k)}=\emptyset\}$ is a~\bol~system of~$(d-1)$-tuples.
Then by the inductive hypothesis, we have
\begin{equation*}
\sum_{A_i^{(k)}=\emptyset}\binom{|A_i^{(1)}|+\cdots+|A_i^{(d)}|}{|A_i^{(1)}|,\ldots,|A_i^{(d)}|}^{-1}\leqslant\frac{1}{d-2}\binom{n+d-3}{d-3}+O(n^{d-4}).
\end{equation*}
Hence
\begin{equation*}
\begin{split}
  \sum_{i=1}^m\binom{|A_i^{(1)}|+\cdots+|A_i^{(d)}|}{|A_i^{(1)}|,\ldots,|A_i^{(d)}|}^{-1} & \leqslant\frac{1}{d-1}\binom{n+d-2}{d-2}+\binom{n+d-3}{d-3}+O(n^{d-4}) \\
    & =\frac{1}{d-1}\binom{n+d-2}{d-2}+O(n^{d-3}).
\end{split}
\end{equation*}
\end{proof}

\section{Proof of Theorem~\ref{Th:d-skew-uniform} and Theorem~\ref{Th:d-space}}

A handy tool to deal with (skew)~\bol~systems of spaces is exterior algebra.
Suppose that~$\alpha_1,\ldots,\alpha_k$ are~$n$-dimensional (raw) vectors in~$\mathbb{R}^n$, let
\begin{equation*}
A=\begin{pmatrix}
\alpha_1 \\
\vdots \\
\alpha_k
\end{pmatrix}
\end{equation*}
be a~$k\times n$ matrix. For~$I\in\binom{[n]}{k}$, denote~$A_I$ the~$k\times k$ submatrix of~$A$ consisting of the~$k$ columns labeled by~$I$.
Then the exterior product (or wedge product) of~$\alpha_1,\ldots,\alpha_k$ is an~$\binom{n}{k}$-dimensional vector
\begin{equation*}
\alpha_1\wedge\cdots\wedge\alpha_k:=(|A_{I}|\mid I\in\binom{[n]}{k}).
\end{equation*}
Note that if~$n=k$ then the exterior product is just the determinant of~$A$.

The above simplified version of the definition of exterior product comes from~\cite{Book}, in where more details and propositions can be found.
For the original abstract definition of exterior product and exterior algebra, one may refer to any textbooks of multilinear algebra.
Here we emphasize the following two lemmas, which are basic conclusions in multilinear algebra, but play a vital role in our proof.

\begin{yinli}
Let~$V=\mathbb{R}^n$, and~$V^k$ be the Cartesian product of~$k$ many of~$V$. Then
\begin{equation*}
f:V^k\rightarrow \mathbb{R}^{\binom{n}{k}},\quad (\alpha_1,\ldots,\alpha_k)\mapsto\alpha_1\wedge\cdots\wedge\alpha_k
\end{equation*}
is a~$k$-linear mapping.
\end{yinli}

\begin{yinli}
Suppose that~$\alpha_1,\ldots,\alpha_k$ are~$k$ many of~$n$-dimensional vectors.
Then they are linearly independent if and only if~$\alpha_1\wedge\cdots\wedge\alpha_k\neq 0$.
\end{yinli}

Suppose~$\alpha=\alpha_1\wedge\cdots\wedge\alpha_s$ and~$\beta=\beta_1\wedge\cdots\wedge\beta_t$, denote
\begin{equation*}
\alpha\wedge\beta:=\alpha_1\wedge\cdots\wedge\alpha_s\wedge\beta_1\wedge\cdots\wedge\beta_t,
\end{equation*}
which is a bilinear mapping.

Fix~$V=\mathbb{R}^n$, and a~$k\leqslant n$. Denote~$\bigwedge^k V$, the collection of the exterior products of~$k$ many of vectors in~$V$.
Then~$\bigwedge^k V$ is a~$\binom{n}{k}$-dimensional vector space.
Let~$W$ be a~$k$-dimensional subspace of~$V$, and~$\alpha_1,\ldots,\alpha_k$ be a basis of~$W$. Denote
\begin{equation*}
\bigwedge W:=\alpha_1\wedge\cdots\wedge\alpha_k\in\bigwedge^k V.
\end{equation*}
Note that this notation is not well defined, for~$\bigwedge W$ depends on the choice of the basis.
However, the only difference between them is a nonzero factor.
So we can use it harmlessly as long as we only care about if the result is zero or not.

To prove Theorem~\ref{Th:d-skew-uniform} and Theorem~\ref{Th:d-space}, we also need the following concept and lemma for general position argument.
Suppose~$V$ and~$W$ are~$n$-dimensional and~$k$-dimensional vector spaces over field~$F$, and~$U_1,\cdots,U_m$ are subspaces of~$V$, where~$\dim(U_i)=r_i$.
We say linear mapping~$\phi:V\rightarrow W$  is in general position with~$U_i$'s, if
\begin{equation*}
\dim(\phi(U_i))=\min\{r_i,k\},\quad i=1,\ldots,m.
\end{equation*}
Note that for every subspace~$U_i$,~$\min\{r_i,k\}$ is the maximum possible dimension of~$\phi(U_i)$.
The following lemma shows that for fixed~$V$,~$W$, and~$U_i$'s, a linear mapping~$\phi:V\rightarrow W$ that in general position with~$U_i$'s always exists,
as long as~$|F|$ is large enough.

\begin{yinli}[\cite{Book}]\label{GP3}
Suppose~$V$ and~$W$ are~$n$-dimensional and~$k$-dimensional vector spaces over field~$F$,~$U_1,\ldots,U_m$ are subspaces of~$V$, and~$\dim(U_i)=r_i$.
If~$|F|>(n-k)(m+1)$ then there exists a linear mapping~$\phi:V\rightarrow W$ that is in general position with~$U_i$'s.
\end{yinli}

Now we are ready to prove Theorem~\ref{Th:d-space}

\begin{proof}[Proof of Theorem~\ref{Th:d-space}]
For every~$k\in[d]$, let~$V_k=\mathbb{R}^{a_1+\cdots+a_k}$. By general position argument (Lemma~\ref{GP3}),
there exists a linear mapping~$\phi_k:V\rightarrow V_k$ that maintains the dimension of~$A_i^{(1)}+\cdots+A_i^{(k)}$ for~$i\in[m]$,
and the dimension of~$A_i^{(p)}+A_j^{(q)}$ for~$i,j\in[m]$ and~$p,q\in[k]$.
The first condition leads to the fact that~$V_k=\phi_k(A_i^{(1)})\bigoplus\cdots\bigoplus\phi_k(A_i^{(k)})$, and the second condition guarantees that
\begin{equation}\label{eq}
\dim(\phi_k(A_i^{(p)})\cap\phi_k(A_j^{(q)}))=\dim(A_i^{(p)}\cap A_j^{(q)}),\quad\forall i,j\in[m],p,q\in[k].
\end{equation}
Actually, we have
\begin{equation*}
\begin{split}
  \dim(\phi_k(A_i^{(p)})\cap\phi_k(A_j^{(q)})) & =\dim(\phi_k(A_i^{(p)}))+\dim(\phi_k(A_j^{(q)}))-\dim(\phi_k(A_i^{(p)})+\phi_k(A_j^{(q)})) \\
    & =\dim(A_i^{(p)})+\dim(A_j^{(q)})-\dim(\phi_k(A_i^{(p)}+A_j^{(q)})) \\
    & =\dim(A_i^{(p)})+\dim(A_j^{(q)})-\dim(A_i^{(p)}+A_j^{(q)}) \\
    & =\dim(A_i^{(p)}\cap A_j^{(q)}).
\end{split}
\end{equation*}

For any~$i\in[m],k\in[d],p\in[k]$,~$\phi_k(A_i^{(p)})$ is a subspace of~$V_k$. Let~$\alpha(i,k,p)=\bigwedge\phi_k(A_i^{(p)})\in\bigwedge^{a_p} V_k$.
For~$i\in[m]$ and~$k\in\{2,\ldots,d\}$, let
\begin{equation*}
f_{i,k}:\bigwedge^{a_k}V_k\rightarrow\mathbb{R},\quad\beta_k\mapsto\left(\alpha(i,k,1)\wedge\alpha(i,k,2)\wedge\cdots\wedge\alpha(i,k,k-1)\right)\wedge\beta_k
\end{equation*}
be a linear mapping. Finally, for every~$i\in[m]$, let
\begin{equation*}
f_i(\beta_2,\ldots,\beta_d):=f_{i,2}(\beta_2)\cdots f_{i,d}(\beta_d),
\end{equation*}
and
\begin{equation*}
\xi_i:=(\alpha(i,2,2),\alpha(i,3,3),\ldots,\alpha(i,d,d))\in\bigwedge^{a_2}V_2\times\cdots\times\bigwedge^{a_d}V_d.
\end{equation*}

Now we claim that~$f_i(\xi_i)\neq 0$ and~$f_i(\xi_j)=0,\forall i<j$.
Actually, for every~$k\in\{2,\ldots,d\}$, we have~$V_k=\phi_k(A_i^{(1)})\bigoplus\cdots\bigoplus\phi_k(A_i^{(k)})$.
Then we can form a basis of~$V_k$ by putting together the bases of~$\phi_k(A_i^{(p)})$ for every~$p\in[k]$. Hence
\begin{equation*}
f_{i,k}(\alpha(i,k,k))=\alpha(i,k,1)\wedge\alpha(i,k,2)\wedge\cdots\wedge\alpha(i,k,k-1)\wedge\alpha(i,k,k)
\end{equation*}
is nonzero, and so~$f_i(\xi_i)\neq 0$. On the other hand, for every~$i<j$, there exist~$p<q$ such that
\begin{equation*}
\dim(\phi_q(A_i^{(p)})\cap\phi_q(A_j^{(q)}))=\dim(A_i^{(p)}\cap A_j^{(q)})>0.
\end{equation*}
Then we have~$\alpha(i,q,p)\wedge\alpha(j,q,q)=0$, so
\begin{equation*}
f_{i,q}(\alpha(i,q,q))=(\alpha(i,q,1)\wedge\alpha(i,q,2)\wedge\cdots\wedge\alpha(i,q,q-1))\wedge\alpha(j,q,q)=0,
\end{equation*}
and~$f_i(\xi_j)=0$.

This is enough to show that~$f_1,\ldots,f_m$ are linearly independent. Actually, suppose that~$c_1f_1+\cdots+c_mf_m\equiv 0$ and~$c_1,\ldots,c_m$ are not all zeros.
Let~$k$ be the maximum index such that~$c_m\neq 0$. Then~$f_m\equiv -\frac{c_1}{c_m}f_1-\cdots-\frac{c_{m-1}}{c_m}f_{m-1}$.
But~$f_m(\xi_m)\neq 0$ and~$f_1(\xi_m)=\cdots=f_{m-1}(\xi_m)=0$ leads to a contradiction.

Note that for fixed~$i\in[m]$, and each~$k\in\{2,\ldots,d\}$,~$f_{i,k}$ is a linear mapping from a~$\binom{a_1+\cdots+a_k}{a_k}$ dimensional space to~$\mathbb{R}$,
and all of this kind of mapping forms a~$\binom{a_1+\cdots+a_k}{a_k}$ dimensional space.
In addition,~$f_i$ belongs to the tensor product of these spaces, so we have
\begin{equation*}
m\leqslant\binom{a_1+a_2}{a_2}\binom{a_1+a_2+a_3}{a_3}\cdots\binom{a_1+\cdots+a_d}{a_d}=\binom{a_1+\cdots+a_d}{a_1,\ldots,a_d}.
\end{equation*}
\end{proof}

Finally, we show how can we obtain Theorem~\ref{Th:d-skew-uniform} from Theorem~\ref{Th:d-space}.

\begin{proof}[Proof of Theorem~\ref{Th:d-skew-uniform}]
Without lose of generality, we may assume~$A_i^{k}\subseteq[n],\forall i\in[m],k\in[d]$ for some~$n$.
Let~$\epsilon_1,\ldots,\epsilon_n$ be a basis of~$\mathbb{R}^n$, and for every set~$A\subseteq[n]$,
let~$s(A)=\spa(\{\epsilon_a\mid a\in A\})$ be a subspace of~$\mathbb{R}^n$.
Consider the family
\begin{equation*}
\mathcal{F}':=\{(s(A_i^{(1)}),\ldots,s(A_i^{(d)}))\mid i\in[m]\},
\end{equation*}
we have~$\dim(s(A_i^{(k)}))=|A_i^{(k)}|=a_k,\forall i\in[m],k\in[d]$, and~$\dim(s(A_i^{(p)})\cap s(A_j^{(q)}))=|A_i^{(p)}\cap A_j^{(q)}|,\forall i,j\in[m],p,q\in[d]$.
Then the fact that~$\mathcal{F}$ is a skew~\bol~system (of~$d$-tuples of sets) guarantees that~$\mathcal{F}'$ is a skew~\bol~system (of~$d$-tuples of spaces). Hence
\begin{equation*}
|\mathcal{F}|=|\mathcal{F}'|\leqslant\binom{a_1+\cdots+a_d}{a_1,\ldots,a_d}.
\end{equation*}
\end{proof}

\begin{center}
\textbf{Remark}
\end{center}

It was several months ago that the study in this paper began, and the counter example of Conjecture~\ref{Conj} (Example~\ref{Ex:counter}) was found.
The first version of this paper was uploaded to arXiv on 2024-11-26.
In the study of Tian and Wu~\cite{Tian}, they proved the invalidness of Conjecture~\ref{Conj} using Young's lattice.
They didn't put their manuscript to arXiv so we didn't notice their work, until Prof. Katona forwarded us their slide used in a conference.
With the help of Young's lattice, they went further on the construction, and they guessed their construction is the maximum case,
but they were not able to prove an upper bound in the general case.
Theorem~\ref{Th:d-nonuniform} in this paper gave an upper bound, and partly answered their question.
Although there is still a gap between the upper bound in Theorem~\ref{Th:d-nonuniform} and the construction in~\cite{Tian},
the combination of these two research shows that their construction (even if not extremum) and our upper bound (even if not tight) are not so bad.
In addition, as I know, both of us discovered Theorem~\ref{Th:d-skew-nonuniform} independently.

\begin{center}
\textbf{Acknowledgments}
\end{center}

We would like to show our appreciate to Prof. Gyula O. H. Katona for let us know the study of Tian and Wu.

\bibliography{bib}

\end{document}